\documentclass{amsart}
\usepackage[margin=1.2in]{geometry}

\usepackage{amsfonts,amssymb,amsmath}
\usepackage{enumerate,enumitem}
\usepackage{verbatim,mathtools,tikz,bm,mathrsfs,tikz-cd,hyperref}
\hypersetup{
    colorlinks=true,
    linkcolor=blue,
    filecolor=magenta,      
    urlcolor=cyan,
}
\usepackage[utf8]{inputenc}
\usepackage[english]{babel}
\usetikzlibrary{matrix}
\usetikzlibrary{arrows} 

\usepackage[all]{xy}
\usepackage{stmaryrd}

\usepackage{xcolor,pgffor,xparse,xstring}
\ExplSyntaxOn
\DeclareExpandableDocumentCommand{\eval}{m}{\int_eval:n {#1}}
\ExplSyntaxOff

\renewcommand{\S}{\mathcal{S}}

\newcommand{\T}{\mathsf{T}}
\newcommand{\D}{\mathsf{D}}
\newcommand{\Df}[1]{\mathsf{D^f}(#1)}
\newcommand{\md}{\mathsf{mod}}

\DeclareMathOperator{\thick}{\mathsf{thick}}

\newcommand{\m}{\mathfrak{m}}
\newcommand{\p}{\mathfrak{p}}

\newcommand{\V}{{\rm V}}
\newcommand{\ov}[1]{\overline{#1}}

\pgfdeclarelayer{bg}    
\pgfsetlayers{bg,main} 
\newcommand{\x}{{\bm{x}}}
\newcommand{\g}{{\bm{g}}}
\newcommand{\del}{\partial}
\newcommand{\f}{{\bm{f}}}

\newcommand{\y}{{\bm{y}}}

\DeclareMathOperator{\pd}{pd}

\DeclareMathOperator{\depth}{depth}

\DeclareMathOperator{\h}{H}

\DeclareMathOperator{\gr}{gr}
\newcommand{\lb}{\llbracket}
\newcommand{\rb}{\rrbracket}

\DeclareMathOperator{\Spec}{Spec}
\DeclareMathOperator{\Supp}{Supp}

\DeclareMathOperator{\Ext}{Ext}
\DeclareMathOperator{\Hom}{Hom}

\DeclareMathOperator{\Kos}{Kos}

\DeclareMathOperator{\supp}{Supp}

\newcommand{\xra}{\xrightarrow}

\newcounter{intro}
\newtheorem{introthm}[intro]{Theorem}

\newtheorem{theorem}[subsection]{Theorem}

\newtheorem{lemma}[subsection]{Lemma}
\newtheorem{corollary}[subsection]{Corollary}

\theoremstyle{definition}
\newtheorem{definition}[subsection]{Definition}
\newtheorem{example}[subsection]{Example}

\newtheorem{notation}[subsection]{Notation}
\newtheorem{strategy}[subsection]{Strategy}
\newtheorem{algorithm}[subsection]{Algorithm}
\newtheorem{fact}[subsection]{Fact}

\newtheorem{remark}[subsection]{Remark}
\newtheorem{question}[subsection]{Question}

\makeatletter
\@namedef{subjclassname@2020}{%
  \textup{2020} Mathematics Subject Classification}
\makeatother

\numberwithin{equation}{subsection}

\begin{document}

\title[Constructing non-proxy small test modules]{Constructing non-proxy small test modules for the complete intersection property}

\author[B.~Briggs]{Benjamin Briggs}
\address{Department of Mathematics,
University of Utah, Salt Lake City, UT 84112, U.S.A.}
\email{briggs@math.utah.edu}

\author[E.~Grifo]{Elo\'{i}sa Grifo}
\address{Department of Mathematics,
University of California, Riverside, CA 92521, U.S.A.}
\email{eloisa.grifo@ucr.edu}

\author[J.~Pollitz]{Josh Pollitz}
\address{Department of Mathematics,
University of Utah, Salt Lake City, UT 84112, U.S.A.}
\email{pollitz@math.utah.edu}



\maketitle

\begin{abstract}
    A local ring $R$ is regular if and only if every finitely generated $R$-module has finite projective dimension. Moreover, the residue field $k$ is a test module: $R$ is regular if and only if $k$ has finite projective dimension. This characterization can be extended to the bounded derived category $\Df{R}$, which contains only small objects if and only if $R$ is regular.
    
    Recent results of Pollitz, completing work initiated by Dwyer-Greenlees-Iyengar,  yield an analogous characterization for complete intersections: $R$ is a complete intersection if and only if every object in $\Df{R}$ is proxy small. In this paper, we study a return to the world of $R$-modules, and search for finitely generated $R$-modules that are not proxy small whenever $R$ is not a complete intersection. We give an algorithm to construct such modules in certain settings, including over equipresented rings and Stanley-Reisner rings. 
\end{abstract}

\vspace{1em}

\section*{Introduction}

Auslander, Buchsbaum and Serre \cite{AuslanderBuchsbaum,Serre} characterized regular local rings in homological terms: a local ring $R$ is regular  if and only if every finitely generated $R$-module has finite projective dimension. Moreover, it is enough to test if the residue field of $R$ has finite projective dimension. The characterization can be phrased in \emph{homotopical} terms, using only the triangulated category structure of the derived category $\D(R)$: $R$ is regular if and only if every complex of $R$-modules with finitely generated homology is quasi-isomorphic to a bounded complex of finitely generated projective $R$-modules, i.e.\ a \emph{perfect complex}, or a \emph{small object} in $\D(R)$.

In \cite{DGI}, Dwyer, Greenlees and Iyengar proposed an analogous characterization for complete intersections. The third author recently settled their question in the positive in  \cite{Pol}, and in turn established a homotopical characterization of complete intersections akin to the homotopical version of the Auslander, Buchsbaum and Serre theorem. Even more recently, the first and third authors, along with Iyengar and Letz (see \cite{BILP}), have provided a new proof, even in the relative case, of this homotopical characterization for complete intersections.

The characterization in \cite{Pol} involves understanding how objects in $\Df{R}$ build small objects; we say that a complex of $R$-modules $M$ \emph{finitely builds} a complex of $R$-modules $N$ if one can obtain $N$ using finitely many cones and shifts and retracts starting from $M$. 

The main result of \cite{Pol} says that $R$ is a complete intersection if and only if every object in $\Df{R}$ finitely builds a nontrivial small object; the forward implication had previously been shown in \cite{DGI}. One can in fact require that each $M$ in $\Df{R}$ builds a perfect complex with the same support as $M$, in which case we say that $M$ is \emph{proxy small}. This should be understood as a weakening of the small property. Since its introduction, the proxy small property has been studied by various authors \cite{Bergh,BILP,DGI2,DGI,GreenleesStevenson,quasipdim,Letz,Pol,Shamir}.

In this paper, our main goal is to complete the picture with a statement involving only \emph{finitely generated $R$-modules}. We aim to show that if every finitely generated $R$-module is proxy small, then $R$ must be a complete intersection. Furthermore, we would like to explicitly construct finitely many test modules $M_1, \ldots, M_t$, playing a role akin to the role that $k$ plays with respect to regularity --- if $R$ is not a complete intersection, one of the $M_i$ must fail to be proxy small. There is another characterization of complete intersections in terms of properties of finitely generated modules --- namely, in terms of finiteness of CI-dimension \cite{AGP}. Finite CI-dimension implies proxy smallness \cite[6.5]{Letz}, but the two notions are not the same --- in particular, the residue field is always proxy small, but its finite CI-dimension is a test for the complete intersection property.

We succeed in this goal when $R$ is equipresented, meaning that given a minimal Cohen presentation $\widehat{R} \cong Q/I$, where $Q$ is a regular local ring, every minimal generator of $I$ has the same $\m$-adic order.

\begin{introthm}\label{introthm} (see Corollary \ref{corep}.) 
For an equipresented local ring $R$,  the following are equivalent:
\begin{enumerate}
    \item $R$ is a complete intersection;
    \item every finitely generated $R$-module is proxy small;
    \item every finite length $R$-module is proxy small.
\end{enumerate}
If the residue field of $R$ is infinite, these are also equivalent to:
\begin{enumerate}
  \setcounter{enumi}{3}
    \item every quotient $R \twoheadrightarrow S$ with $S$ an artinian hypersurface is proxy small.
\end{enumerate}
\end{introthm}

For equipresented rings, Theorem \ref{introthm} strengthens the  characterization of complete intersections established in \cite{Pol} (and \cite{BILP}) with a new proof.  We also give an algorithm to find quotients of $R$ that are artinian hypersurfaces but not proxy small, whenever $R$ is not a complete intersection and has infinite residue field.

Equipresented rings are part of a larger class of rings for which the theorem holds (cf.~Theorem \ref{t:characterization}), which are said to have \emph{large enough cohomological support} (see Definition \ref{definition large support}). 
Over such rings, provided the residue field is infinite,

\vspace{1em}
\begin{quote}
$R$ is a complete intersection if and only if every surjection to an artinian hypersurface is proxy small.  
\end{quote}
\vspace{1em}

We expect this property to characterize complete intersections among all local rings, analogously to the characterization of regular local rings in terms of their residue fields. This remains open in general. 

For all equipresented rings, and more generally all rings of large enough cohomological support, we also answer a question of Gheibi, Jorgensen, and Takahashi \cite[Question 3.9]{quasipdim}. This question proposes yet another characterization of complete intersections: that $R$ is a local complete intersection if and only if every finitely generated $R$-module has finite quasi-projective dimension (see the end of Section \ref{section main result} and \cite{quasipdim} for a definition and other details).

In the last section we construct explicit modules that are not proxy small over various rings.  Our methods are not limited to equipresented rings, and we include here all Stanley-Reisner rings (see Example \ref{example monomial}) and short Gorenstein rings (see Example \ref{exampleshortgor}). 

In Sections \ref{section proxy small} and \ref{section support}, we recall the definition and basic properties of proxy small modules and of cohomological support, respectively. In Section \ref{section main result} we prove our main result, and give an algorithm for finding modules that are not proxy small. In Section \ref{section examples} we apply the results of the previous section in various examples.

\section{Proxy small objects}\label{section proxy small}
Let $R$ be a commutative noetherian ring. We let $\D(R)$ denote its derived category of (left) $R$-modules (see \cite[1.2]{krause} for more on the derived category). Much of the homological information of $R$ is captured by how several of its subcategories are related to each other using the triangulated category structure of $\D(R)$. In what follows we clarify this point and introduce the main objects of interest. First, we need some terminology.

\begin{definition}
A \emph{thick subcategory} of $\D(R)$ is a full subcategory $\T$ that is closed under taking shifts, cones, and direct summands. That is, if $X=X'\oplus X''$ is an object of $\T$, then $X'$ and $X''$ are objects of $\T.$ The smallest thick subcategory of $\D(R)$ containing an object $M$ of $\D(R)$ is denoted $\thick M$ and is called the \emph{thick closure of $M$}; this exists since an intersection of thick subcategories is again thick. Alternatively, one can define $\thick M$ inductively as in \cite{HPC}. 
\end{definition}

In the terminology used in the Introduction, $M$ finitely builds $X$ precisely when $X$ is in the thick closure of $M$.

\begin{example}\label{ex:fg}
The full subcategory of $\D(R)$ consisting of objects $M$ such that $\h(M)$ is a finitely generated graded $R$-module forms a thick subcategory of $\D(R)$, since $R$ is noetherian. This category will be denoted $\Df{R}.$ The category of finitely generated $R$-modules, denoted $\md(R)$, sits inside $\Df{R}$ as a full subcategory by including each finitely generated $R$-module as a complex concentrated in degree zero; however, $\md(R)$ is not a thick subcategory (it fails even to be triangulated).
\end{example}

\begin{example}\label{ex:small}
The full subcategory of $\D(R)$ consisting of complexes of $R$-modules that are quasi-isomorphic to a bounded complex of finitely generated projective $R$-modules forms a thick subcategory of $\D(R).$ Moreover, this category is exactly $\thick R$ and its objects will be called \emph{small}; note that these are also referred to as \emph{perfect} complexes \cite{HPC}, but we have opted for the former terminology since it describes these complexes categorically as objects of $\D(R)$. Namely, for each small object $M$, $\Hom_{\D(R)}(M,-)$ commutes with arbitrary (set-indexed) direct sums \cite[3.7]{DGI}. 
\end{example}

\begin{fact}\label{c:reg}
The relation between the categories discussed in \ref{ex:fg} and \ref{ex:small} can be used to detect the singularity of $R$. Namely, the following are equivalent:
\begin{enumerate}
\item $R$ is regular (meaning that each localization is a regular local ring);
\item each object of $\Df{R}$ is small;
\item each object of $\md(R)$ is small;
\item each residue field of $R$ is small.
\end{enumerate}
This is essentially the Auslander-Buchsbaum-Serre theorem \cite{AuslanderBuchsbaum,Serre} combined with a local-to-global result of Bass and Murthy \cite[4.5]{BM} (see also \cite[4.1]{AIL}).
\end{fact}

The analogous characterization for locally complete intersections is in terms of proxy small objects of $\D(R)$. These were introduced and studied by Dwyer, Greenlees and Iyengar in \cite{DGI2,DGI}. To define them, we recall that the support of a complex $X$ is $\supp_R X\coloneqq \{\p\in \Spec R: X_\p\not\simeq 0\}$, extending the usual notion for $R$-modules.

\begin{definition}
A complex of $R$-modules $M$ is \emph{proxy small} if $\thick M$ contains a small object $P$ such that $\supp_R P=\supp_R M$.
\end{definition}

\begin{remark}
Let $M$ be a proxy small object. It follows easily that the support of $M$ is a closed subset of $\Spec R$ \cite[Proposition 4.4]{DGI}. Furthermore, as a consequence of a theorem of Hopkins and Neeman \cite{H,N}, the object $P$ witnessing $M$ as proxy small can be taken to be the Koszul complex on an ideal $I$ defining $\supp_R M$, i.e.,
\[\supp_R (M)=\supp_R (R/I).\] 
In particular, when $R$ is local and $M$ has finite length homology, $M$ is proxy small if and only if $\thick M$ contains the Koszul complex on a list of generators for the maximal ideal of $R$. 
\end{remark}

\begin{fact}\label{c:lci}
In a similar fashion to Fact \ref{c:reg}, the following are equivalent:
\begin{enumerate}
    \item $R$ is locally a complete intersection, meaning that each localization is a local complete intersection;
    \item each object of $\Df{R}$ is proxy small.
\end{enumerate}
This is \cite[5.2]{Pol} combined with a local-to-global result of Letz \cite[4.5]{Letz}. This can also be recovered by recent work in \cite{BILP}. One of the main points of this article is to fill in the missing analogous conditions to conditions (3) and (4) from \ref{c:reg}. As mentioned previously, the difficulties arise as $\md(R)$ does not respect the triangulated structure of $\Df{R}$. 
\end{fact}

\section{Cohomological support for local rings}\label{section support}

In this section we review the necessary theory of cohomological supports over a local ring; see \cite[Section 2]{Jor}, \cite[Section 3]{Pol} and \cite[Sections 4 \& 5]{Pol2} for further details.  This theory offers a method to detect thick subcategory containment, an idea that goes back to \cite{H,N}. In particular, its utility is in showing that an object cannot be proxy small.    

The theory of cohomological support utilized in the present article originated  in the pioneering work of Avramov \cite{VPD} and in his collaboration with  Buchweitz \cite{SV}; these supports were defined over complete intersections and were  successfully linked  to  cohomological information  ultimately revealing remarkable symmetries in the asymptotic information of Ext and Tor modules over a complete intersection. The varieties were later extended and studied outside of the realm of complete intersections in \cite{AIRestricting, BW, Jor,Pol,Pol2}. As it was shown in \cite{Pol2}, these theories of supports are all recovered by the cohomological support in \cite{Pol2}. For this article, we take the definition from \cite{Jor} (or more generally, the one from \cite{AIRestricting}) while exploiting some of the properties from \cite{Pol,Pol2} (see Fact \ref{var}). 

We fix once and for all a local ring $R$ along with a minimal Cohen presentation 
\[
\widehat{R}\cong Q/I,
\] 
so $(Q,\m,k)$ is a regular local ring and $I \subseteq \m^2$. 

\begin{definition}[See \cite{Jor,AIRestricting}]
We define $\V_R$ to be the vector space $I/\m I$.  
For a finitely generated $R$-module $M \neq 0$, we define the \emph{cohomological support of $M$} to be
\[
\V_R(M):= \left\{ [f] \in \V_R\ \left|\right.\  \pd_{Q/f} \widehat{M} = \infty \text{ or }  [f] = 0 \right\}.
\]
By convention, $\V_R(0)$ is empty.
\end{definition}

The vector space $\V_R$ is also known as $\pi_2(R)$, or, in older literature, as $\V_2(R)$ (cf.~\cite{VPD}). Note that $\V_R$ is intrinsic to $R$, and does not depend on our choice of a minimal Cohen presentation, since in fact the definition of cohomological support we are using coincides with that of \cite[5.2.4]{Pol2}, which is independent of the choice of Cohen presentation \cite[5.1.3]{Pol2}.

\begin{remark}
We briefly indicate an alternative perspective on the cohomological support, explained in more detail in \cite{Pol,Pol2}. Let $\S = k[\V_R]$ be the graded ring of polynomials on $\V_R$, generated by $(\V_R)^*$ in degree $2$. By definition, $\V_R$ identifies with the set of $k$-points in $\Spec\S$. Let $E={\rm Kos}^Q(\f)$ be the Koszul complex on some minimal generating set $\f$ for $I$, regarded as a dg $Q$-algebra in the usual way. 
By \cite[2.4]{CD2}, there is a natural inclusion $\S\subseteq \Ext_E(k,k)$ 
making $\Ext_E(k,k)$ a flat, module-finite $\S$-algebra;  when $\f$ is a $Q$-regular sequence $\S$ agrees with the cohomology operators of Gulliksen \cite{G} and Eisenbud \cite{Eis}, up to sign \cite{AS}.
In any case, it follows that $\Ext_E(\widehat{M},k)$ is a finitely generated graded $\S$-module for each $M$ in $\Df{R}$ \cite[3.2.4]{Pol}. Finally, the set $\V_R(M)$ defined above can be identified with the $k$-points of the reduced subscheme $\Supp_\S \Ext_E(\widehat{M},k) \subseteq \Spec \S$ \cite[5.2.4]{Pol2}. From this we deduce that $\V_R(M)$ is a Zariski closed, conical subset of $\V_R$.
\end{remark}

We will explicitly use the \emph{vector space} structure of $\V_R$, so our definition of cohomological support is the most convenient in this context. 

The invariant $\V_R(R)$ is interesting in its own right; besides detecting the complete intersection property, as we note below, it contains more information about the structure of $R$ in general.

\begin{fact}\label{var}
We recall two facts regarding these cohomological supports (see \cite[3.3.2]{Pol}).
\begin{enumerate}
\item \label{var1} If $M$ is a proxy small object of $\Df{R}$, then $\V_R(R) \subseteq \V_R(M)$. 
\item \label{var2} $\V_R(R)=0$ if and only if $R$ is a complete intersection.
\end{enumerate}
\end{fact}

\begin{strategy}\label{strategy}
Recall that our primary goal is to show that if $R$ is not a complete intersection then there are finitely generated $R$-modules that are not proxy small, which we will ultimately do in Theorem \ref{t:characterization}. We isolate, and slightly modify, the strategy from \cite[5.2]{Pol}. Our goal in the present paper is to provide an explicit list of  \emph{finitely generated $R$-modules} $M_1,M_2,\ldots,M_t$ satisfying \[\V_R(R)\not\subseteq \V_R(M_1)\cap \V_R(M_2)\cap \ldots \cap \V_R(M_t)\] provided that  $R$ is not a complete intersection. When  such modules $M_1,M_2,\ldots, M_t$ exist, Fact \ref{var} implies at least one $M_i$ fails to be proxy small. 
\end{strategy}

In the next result, and below, we will need some notation:

\begin{notation}
Let $R$ be a local ring with residue field $k$. For any local homomorphism $\varphi\colon R\to S$, there is an induced map of $k$-vector spaces
\[
\V_\varphi \colon \V_R \longrightarrow \V_S,
\]
constructed as follows. Completing if necessary, one can choose Cohen presentations $R=Q/I$ and $S=Q'/J$ and a compatible lift $\widetilde{\varphi}\colon (Q,\m)\to (Q',\m')$ of $\varphi$, see \cite{AvramovFoxbyHerzog}. In particular $\widetilde{\varphi}(I)\subseteq J$, so there is an induced map of $k$-vector spaces
\[
I/\m I\longrightarrow J/\m'J,
\]
which we denote by $\V_\varphi$.
\end{notation}

We now prove a key lemma, which gives us an explicit formula for the cohomological support in a very specific but important case.

\begin{lemma}
\label{l:ben}
Let $(Q,\m)\to (Q',\m')$ be a finite flat extension of regular local rings such that $\m Q'=\m'$, inducing a map $\varphi\colon R=Q/I \to S=Q'/J$. If $J$ is generated by a $Q'$-regular sequence, then $\V_R(S)=\ker (\V_\varphi)$.
\end{lemma}

\begin{proof}
Unraveling the notation, the claim is that for any $f \in I$, we have
\[
\pd_{Q/f}(Q'/J)=\infty \text{ if and only if } f \in \m J=\m' J \textrm{ in } Q'.
\]
If $f\in J\smallsetminus \m' J$, then it forms part of a regular sequence $f,g_2,...,g_m$ generating $J$, and the Koszul complex 
\[
A=\Kos^{Q'/f}(g_2, \ldots, g_m)=Q'/f\langle x_2,\ldots, x_m \mid \del x_i=g_i\rangle
\]
is a finite free resolution of $Q'/J$ over $Q'/f$ (the latter notation is explained in \cite{GL} or \cite[Section 6]{IFR}, for example). But $Q'/f$ is free over $Q/f$, so this is also a finite free resolution over $Q/f$, and the forward implication holds.

If on the other hand $f \in \m'J$, then we may write $f=\sum a_i g_i$ with $a_i \in \m'$ and $g_1,..., g_n$ a regular sequence generating $J$. We can then form the Tate model
\[
{\rm Kos}^{Q'/f}(g_1,...,g_m)\langle y \rangle = A\langle y\mid \del (y) = {\textstyle \sum a_ix_i}\rangle,
\]
where the $x_i$ are the degree one Koszul variables with $\partial(x_i)=g_i$ . By \cite[Theorem 4]{Tate}, a minimal free resolution of $Q'/J$ over $Q'/f$ is  \[{\rm Kos}^{Q'/f}(g_1,...,g_m)\langle y \rangle\xra{\simeq} Q'/J\,.\] Finally, as $\m Q'=\m'$ it is also minimal as a complex of free $Q/f$ modules, we conclude that $\pd_{Q/f}(Q/J)=\infty$.
\end{proof}

\begin{remark}
\label{r:ben}
In particular, if $J$ is an ideal of $Q$ generated by a regular sequence and $I\subseteq J$, then 
$$\V_R(Q/J)=\ker \left( I/\m I \to J / \m J\right).$$
\end{remark}

\section{Main result}\label{section main result}

\begin{lemma}\label{l:primeavoidance}
Let $S=k[x_1,\ldots, x_e]$ be a standard graded polynomial ring over an infinite field $k$. For a homogeneous ideal $I$ of $S$, if $h$ is a homogeneous generator of $I$ of minimal degree, then there exists a regular sequence of linear forms $\bm{\ell}=\ell_2,\ldots, \ell_e$ in $S$ such that $h$ is a nonzero element of minimal degree in $I(S/\bm{\ell}).$
\end{lemma}

\begin{proof}
As $h$ is homogeneous, $S/h$ is a standard graded $k$-algebra of dimension $e-1$. By \cite[1.5.17]{BH}, since $k$ is infinite there exists an algebraically independent system of parameters $\g=g_2,\ldots, g_e$ for $S/h$ such that each $g_i$ has degree 1. Let $\ell_i$ be a lift of $g_i$ back to a linear form of $S$. Since $h,\bm{\ell}$ is a homogeneous system of parameters for $S$, it follows that $h$ is nonzero in $S/\bm{\ell}.$ The only thing left to remark is that since $\bm{\ell}$ consists of linear forms (in fact, homogeneous is enough), the image of $h$ still has minimal degree among homogeneous elements of $I (S/\bm{\ell}).$
\end{proof}

The next lemma establishes the existence of certain complete intersection quotients that are defined by exactly one of the defining relations for the given local ring. Ultimately, these quotients are the ones that will serve as the sought after test modules, provided that $R$ has enough of them. 

\begin{lemma}\label{l:testmodule}
Let $R$ be a local ring with fixed minimal Cohen presentation $\widehat{R}=Q/I$, where $(Q,\m,k)$ is regular and $k$ is infinite. For any $f\in I\smallsetminus \m I$ with minimal $\m$-adic order among elements of $I$, there exists a singular artinian hypersurface $S$ which is a quotient of $R$, say $R \twoheadrightarrow Q/J \cong S$, such that $f\in J\smallsetminus \m J$.
\end{lemma}

\begin{proof}
Let $e$ be the embedding of $R$, that is, the Krull dimension of $Q$.  Fix a $Q$-regular sequence $\x=x_1,\ldots, x_e$ generating $\m$. As $\x$ is $Q$-regular, the associated graded ring of $Q$
\[
\gr Q\coloneqq \bigoplus_{i=0}^\infty \m^i/\m^{i+1},
\] is a standard graded polynomial ring on $\gr x_1,\ldots ,\gr x_e$, the image of $\x$ in $\gr Q$, over $k$. 

Applying Lemma \ref{l:primeavoidance} with $h=\gr f$, we obtain a regular sequence of linear forms $\bm{\ell}=\ell_2,\ldots, \ell_e$ in $\gr Q$ such that the image of $\gr f$ is nonzero and has minimal degree in $\gr Q/(\bm{\ell})$ among all    elements of $\gr I.$ The sequence $\bm{\ell}$ determines a sequence $\y=y_2,\ldots,y_e $ in $Q$ such that the image of $\y$ is linearly independent in $\m/\m^2$ and $\gr y_i=\ell_i$ for each $i$. 

Let $\ov{( \ \ )}$ denote reduction modulo $(\y)$. Since $Q$ is a regular local ring and $\y$ is a regular sequence in $\m/\m^2$, there is an isomorphism of graded $k$-algebras 
\begin{equation}\label{e:gradediso}
\gr Q/(\ell)\cong \gr (\ov{Q})
\end{equation}  such that the degree of $\gr(\ov{f})$ is exactly the $\ov{\m}$-adic order of $\ov{f}$. By the assumptions on the image of $\gr f$ in $\gr Q/(\ell)$ and the isomorphism of graded $k$-algebras in (\ref{e:gradediso}), it follows that $\ov{f}$ is a nonzero element of $\ov{Q}$ that has minimal degree among elements of $I\ov{Q}.$ However, $\ov{Q}$ is  a DVR and so $(\ov{f})=I\ov{Q}.$ 

So, setting $J=(f,y_2,\ldots, y_e)$, we have shown that $J$ contains $I$ and has $f$ as a minimal generator, and also that $Q/J$ is an artinian hypersurface. Furthermore, as $f$ was part of a defining system for a minimal Cohen presentation, the $\m$-adic order of $f$ is at least two; thus, $Q/J$ is non-regular. 

Finally, since $Q/J$ is artinian, the composition $R\to \widehat{R}\to Q/J$ is surjective.
\end{proof}

    Fix a local ring $R$ with minimal Cohen presentation $(Q,\m,k)\xra{\pi} \widehat{R}$ and let $I$ denote $\ker \pi$. In Lemma \ref{l:testmodule}, we showed the existence of artinian hypersurfaces that are quotients of $R$ defined by a \emph{minimal} relation of $\widehat{R}$. 
    When $R$ has ``enough" of these hypersurface quotients, we can appeal to Strategy \ref{strategy} with these quotients serving as our list of test modules. 
    The condition below guarantees that $R$ has ``enough" of the quotients from Lemma \ref{l:testmodule}.

\begin{definition}\label{definition large support}
We say that a local $R$ with minimal Cohen presentation \[ \widehat{R}\cong (Q,\m,k)/I\] has \emph{large enough cohomological support} provided that
\begin{equation}\label{e:technical} 
\dim_k \left( \dfrac{\m^{d+1}\cap I}{\m I} \right) < \dim_k \left( {\rm span}\V_R(R) \right),
\end{equation} 
where $d$ denotes the order of $I$, meaning the minimal $\m$-adic order of an element of $I$ in the regular local ring $Q$.
\end{definition}

\begin{remark}\label{r:enoughpstestmodules} 
If $R$ has large enough cohomological support, then it is not a complete intersection (see \ref{var}(\ref{var2})). The defining condition (\ref{e:technical}) means that there exists a minimal generating set
\[
\{f_1,\ldots, f_c,f_{c+1},\ldots,f_n\}
\]
for the ideal $I$, where every $k$-linear combination of $f_1,\ldots, f_c$ has minimal $\m$-adic order among elements of $I$, and
\[
n-c<\dim_k ({\rm span} \V_R(R)).
\]
Recall that by \cite{Pol}, $\V_R(R)$ is zero whenever $R$ is a complete intersection.
\end{remark}

As we will see, if $R$ has large enough cohomological support $\V_R(R)$, that does not mean that $\V_R(R)$ is necessarily a large set, rather that it's ``large enough" for us to be able to find the modules we are looking for. There are two extremes among non-complete intersections which are easily seen to satisfy this condition, as we will see in Example \ref{exampleequipresented} and Example \ref{examplemaximalsupport}.

\begin{example}[Equipresented rings]\label{exampleequipresented}
Suppose that a minimal Cohen presentation $Q\xra{\pi} \widehat{R}$ of $R$ is such that every minimal generator of $I = \ker \pi$ has the same $\m$-adic order; we say such a ring is \emph{equipresented}. If $R$ is not a complete intersection, then (\ref{e:technical}) is satisfied trivially. These include:
\begin{enumerate}
        \item Short Gorenstein rings (see Example \ref{exampleshortgor}), Veronese rings of polynomial rings, and indeed all Koszul algebras.
        \item More generally, any quadratic ring.
        \item The truncated rings $Q/\m^d$ (see Example \ref{example truncated}).
        \item Generic determinantal rings.
    \end{enumerate}
\end{example}

\begin{example}[Rings with spanning support]\label{examplemaximalsupport}
Let $n$ denote the minimal number of generators for a defining ideal $I$ of $\widehat{R}$ in a minimal Cohen presentation of $(Q,\m)\xra{\pi} \widehat{R}$.
Assume $R$ is not a complete intersection and that it satisfies $\dim_k ({\rm span} \V_R(R))=n$. Such rings are said to have \emph{spanning support}.

If $d$ denotes the minimal $\m$-adic order of an element in $I$, then
\[
\dim_k \left( \dfrac{\m^{d+1}\cap I}{\m I} \right)<n.\]
Hence, (\ref{e:technical}) is trivially satisfied for rings with spanning support. Here are some examples of rings with spanning support: 
    \begin{enumerate}
        \item Short Gorenstein rings (see Example \ref{exampleshortgor}). 
        \item \label{examplesmallcodepth} By \cite[5.3.6]{Pol2}, if $R$ is not a complete intersection and $\dim Q-\depth R\leqslant 3$, then 
        \[
        \V_R(R)=\V_R
        \] 
        except when $Q$ admits an embedded deformation, meaning that $\widehat{R} \cong P/(f)$ for some local ring $P$ and some $P$-regular element $f$. So the generic non-complete intersection which satisfies $\dim Q-\depth R\leqslant 3$ has spanning support, and hence has large enough cohomological support. 
\item Suppose $(Q,\m)\to \widehat{R}$ is a minimal Cohen presentation of $R$ where the minimal free resolution $F\xra{\simeq} \widehat{R}$ admits a DG $Q$-algebra structure. If $F_1F_1\subseteq \m F_2$ then $R$ has spanning support (see the argument in \cite[5.3.3]{Pol2}). Similarily, a direct calculation shows that  if $F_1F_{p-1}\subseteq \m F_p$ where $p=\pd_Q \widehat{R}$, then $R$ again has spanning support. 
    \end{enumerate}
\end{example}

Here is a procedure to construct examples that are not in either of the two classes above, in Examples \ref{exampleequipresented} and \ref{examplemaximalsupport}, yet have large enough cohomological support. 

\begin{example}
    Let $R'$ be a non-complete intersection local $k$-algebra which is equipresented in degree $d$ and let $S$ be a complete intersection $k$-algebra whose defining ideal is generated in degrees strictly greater than $d$. Set 
    \[
    R\coloneqq R'\otimes_k S.
    \] 
    It can easily be checked that $R$ has large enough cohomological support while not falling into the two classes of rings above. 
        
     Furthermore, the tensor product (over $k$) of any two non-complete intersections, each without spanning support and not equipresented, yields a $k$-algebra with not falling into either class above; this is a fairly large class of examples of rings with large enough cohomological support that are not in the extremal cases of Examples \ref{exampleequipresented} and \ref{examplemaximalsupport} above. 
\end{example}

\begin{theorem}\label{t:characterization}
Let $R$ be a commutative noetherian local ring with residue field $k$. If $R$ has large enough support, 
 then there exists a finite length $R$-module that is not proxy small. Moreover, when $k$ is infinite there is a surjective homomorphism $R \twoheadrightarrow S$ such that $S$ is an artinian hypersurface and $S$ is  not a proxy small $R$-module. 
\end{theorem}

\begin{proof}
Consider a minimal Cohen presentation $(Q,\m)\xra{\pi} \widehat{R}$ with kernel $I$. Since $R$ has large enough support, by Remark \ref{r:enoughpstestmodules} $I$ is generated by 
\[
\{f_1,\ldots, f_c,f_{c+1},\ldots,f_n\}
\]
 where each $k$-linear combination of $f_1,\ldots, f_c$ has minimal $\m$-adic order among elements of $I$, and
\[
n-c<\dim_k \left({\rm span} \V_R(R)\right).
\]
First we assume that $k$ is infinite. With this set up we build $t \leqslant c$ artinian hypersurface quotients $R \twoheadrightarrow Q/J_i$ for $1 \leqslant i \leqslant t$ such that  
\[
\V_R(Q/J_1)\cap \V_R(Q/J_2)\cap \ldots \cap \V_R(Q/J_t)
\]
is a subspace of codimension at least $c$ in $\V_R$, and thus cannot contain $\V_R(R)$. We will then use Strategy \ref{strategy} to conclude that at least one of these artinian hypersurface quotients cannot be proxy small. 

First, Lemma \ref{l:testmodule} provides an artinian hypersurface $R \twoheadrightarrow Q/J_1$ such that $g_1 := f_1 \in J_1 \smallsetminus\m J_1$. In particular, by Lemma \ref{l:ben}, $\V_R(Q/J_1)$ is a subspace of $I/\m I$ of codimension at least $1$. If possible, pick a minimal generator $g_2 := a_1 f_1 + \cdots + a_c f_c$ of $I$ such that $g_2 \in \m J_1$; note that $g_2$ has minimal $\m$-adic order by construction. If there is no such $g_2$, then $\V_R(Q/J_1)$ is a subspace of codimension at least $c$, and we are done. If such a $g_2$ does exist, then we again use Lemma \ref{l:testmodule} to build an artinian hypersurface quotient $R \twoheadrightarrow Q/J_2$ such that $g_2 \in J_2 \smallsetminus\m J_2$. Note that the codimension of $\V_R(Q/J_1) \cap \V_R(Q/J_2)$ must necessarily increase, by construction, and so in particular it is at least $2$. Proceeding by induction, we build artinian quotients $Q/J_1, \ldots, Q/J_t$ of $R$, $t \leqslant c$, such that $\bigcap_{j=1}^t \V_R(Q/J_j)$ 
is a subspace of $I/\m I$ of codimension at least $c$. From the assumption that $R$ is not a complete intersection and $k$ is infinite, we have constructed a finite length non-proxy small $R$-module. 

Now we deal with the case when $k$ is finite. By \cite[Appendice, \textsection 2]{Bourbaki} (see also \cite[Theorem 10.14]{CMrepsbook}) one can construct a flat extension of regular local rings $(Q,\m,k)\to (Q',\m',k')$ such that $\m Q'=\m'$ and such that $k'$ is infinite. Moreover, choosing $k'$ to be algebraic over $k$, we can do this in such a way that $Q'$ is a colimit of regular local rings $(Q_i,\m_i,k_i)$ each finite and flat over $Q$, each satisfying $\m Q_i=\m_i$.

Since $\m Q'=\m'$, the $\m$-adic order of an element of $Q$ is the same as its $\m'$-adic order in $Q'$. So applying the above argument to $IQ'$ yields a sequence of ideals $J_1',\dots,J_t'$ of $Q'$ such that
\[
\bigcap_{j=1}^t \ker\left( I/\m I \to J_j'/\m'J_j' \right)
\]
has codimension at least $c$ in $\V_R=I/\m I$.

There is an $i$ such that all of the ideals $J_1,\dots,J_t$ are defined over $Q_i$. In other words, we can find ideals $J_{1},\dots,J_{t}$ of $Q_i$ for which $J_{j}Q'=J_j'$. Now applying Lemma \ref{l:ben}, we see that
\[
\bigcap_{j=1}^t \V_R(Q_i/J_i)=\bigcap_{j=1}^t \ker\left( I/\m I \to J_j/\m J_j \right)=\bigcap_{j=1}^t \ker\left( I/\m I \to J_j'/\m'J_j' \right)
\]
has codimension at least $c$ in $\V_R$; the second equality here uses flatness of $Q_i\to Q'$. Finally, we can conclude as above that one of $Q_i/J_i$ must fail to be proxy small as an $R$-module.
\end{proof}

\begin{remark}
The finite length modules that are constructed in Theorem \ref{t:characterization} are shown to exist based on the specified ring theoretic information in condition (\ref{e:technical}); the latter property is on the $\m$-adic order of elements in $I/ \m I$, where $(Q,\m)\to Q/I$ is a minimal Cohen presentation for $R$. Moreover, these finite length modules are shown to exist for any singular local ring $R$ since   $n-c>0$ where $n$ and $c$ are as in Remark \ref{r:enoughpstestmodules}. So this set of modules acts as a list of test modules provided condition (\ref{e:technical}) in Theorem \ref{t:characterization} holds. However, it remains to determine whether one can construct a canonical list of finitely generated modules that detect whether $R$ is a complete intersection as discussed in Strategy \ref{strategy}; it may be worth exploring this idea further.
\end{remark}

\begin{remark}
One natural guess for a test module would be the conormal module $I/I^2$. While proxy small modules seem to govern the complete intersection property, so does the conormal module \cite{Vasconcelos,AvramovHerzog,Ben}. In particular, \cite{Ben} succeeded in showing that smallness of the conormal does in fact chracterize a complete intersection; see also \cite{BI}. However, the conormal module is often proxy small even if $R$ is not a complete intersection: if $R$ admits an embedded deformation, then $I/I^2$ has a free summand \cite{Vasconcelos}, and thus $I/I^2$ is proxy small.
\end{remark}

As a special case of Theorem \ref{t:characterization}, we can completely solve the case of equipresented rings.

\begin{corollary}\label{corep}
For an equipresented local ring $R$ with residue field $k$, the following are equivalent:
\begin{enumerate}
    \item\label{t:ep1} $R$ is a complete intersection;
    \item\label{t:ep2} every finitely generated $R$-module is proxy small;
    \item\label{t:ep3} every finite length $R$-module is proxy small.
\end{enumerate}
If $k$ is infinite, then these are equivalent to:
\begin{enumerate}
  \setcounter{enumi}{3}
    \item\label{t:ep4} for every surjective homomorphism $R \twoheadrightarrow S$ such that $S$ is an artinian hypersurface, $S$ is a proxy small $R$-module.
\end{enumerate} 
\end{corollary}

In fact, the proof of Theorem \ref{t:characterization} provides an algorithm to find modules that are not proxy small.

\begin{algorithm}\label{alg}
Suppose that $R \cong Q/I$, where $Q$ is a regular ring and $I = (f_1, \ldots, f_n)$ is such that every $k$-combination of $f_1, \ldots, f_n$ has the same $\m$-adic order.

\begin{enumerate}
\item[Step 1] Find $x_2, \ldots, x_e \in \m \smallsetminus \m^2$ regular on $R/(f_1)$, and set $J_1 := (f_1, x_2, \ldots, x_e)$ and $M_1 := R/J_1$. As we have shown in Theorem \ref{t:characterization}, the ideal $J_1$ contains $I$.

\item[Step 2] Compute 
$$K_1 = \ker \left( \dfrac{I}{\m I} \to \dfrac{J_1}{\m J_1}\right).$$ 
Note that this map is well-defined because $I \subseteq J_1$.

\item[Step 3] For a fixed $r \geqslant 1$, suppose we have constructed $J_1, \ldots, J_r$ and $K_1, \ldots, K_r$. Check if there is an equality $K_1 \cap \cdots \cap K_r = 0$; if so, we are done. If not, take $g_{r+1} \in R$ of minimal $\m$-adic order such that $[g_{r+1}] \in K_1 \cap \cdots \cap K_r$. Repeat Step 1 for $g_{r+1}$, that is, find $y_2, \ldots, y_e \in \m \smallsetminus\m^2$ such that $g_{r+1}, y_2, \ldots, y_e$ is a regular sequence. Set $J_{r+1} = (g_{r+1}, y_2, \ldots, y_e)$ and $M_{r+1} = Q/J_{r+1}$. Repeat also step 2, by setting 
$$K_{r+1} := \ker \left( \dfrac{I}{\m I} \to \dfrac{J_{r+1}}{\m J_{r+1}}\right).$$
\end{enumerate}

\noindent In each step, the dimension of the vector space $K_1 \cap \cdots \cap K_r$ goes down by at least 1; therefore, the process stops after at most $n$ steps, since $[f_1] \notin K_1$ and thus $K_1$ has dimension at most $n-1$. Once this process is completed, we are left with $R$-modules $M_1, \ldots, M_{t}$ such that $\V_R(M_1)\cap \ldots \cap \V_R(M_t) = 0$. If $R$ is not a complete intersection, at least one of the $M_i$ cannot be proxy small.

When $Q$ is a polynomial ring over a field $k$ and $R$ is a quotient of $Q$ by some homogeneous ideal $I$, this can be done with the computer algebra system Macaulay2 \cite{M2}. If $f_1$ is a homogeneous generator of minimal degree in $I$, the method \texttt{inhomogeneousSystemOfParameters} from the \texttt{Depth} package will find a linear system of parameters $[x_2], \ldots, [x_e]$ in $R/(f_1)$, and thus $f_1, x_2, \ldots, x_e$ form a regular sequence. Moreover, as shown in Theorem \ref{t:characterization}, $I \subseteq J_1 := (f_1, x_2, \ldots, x_e)$.
\end{algorithm}

We will apply this algorithm in Section \ref{section examples} to compute various examples.

\begin{remark}
    If $I$ is not equigenerated, but still satisfies the hypothesis of Theorem \ref{t:characterization}, a variation of Algorithm \ref{alg} will still produce our candidates for non-proxy small modules. Suppose that $I$ is a homogenous ideal in $Q = k[x_1, \ldots, x_v]$, minimally generated by $n$ elements, and that
$$n - r := \dim_k \left( \dfrac{(x_1, \ldots, x_v)^{d+1} \cap I}{(x_1, \ldots, x_v) I} \right) < \dim_k \left( {\rm span}\V_R(R) \right) := s.$$
Find homogeneous generators $f_1, \ldots, f_n$ for $I$ such that every $k$-combination of $f_1, \ldots, f_r$ has minimal degree in $I$, and such that $f_{r+1}, \ldots, f_n$ have non-minimal degree in $I$. We then run Algorithm \ref{alg} on $(f_1, \ldots, f_r)$,  but rather than checking at each step that $K_1 \cap \cdots \cap K_t = 0$, we check that 
\[
\dim_k \left( K_1 \cap \cdots \cap K_t \right) < s.
\]
We use this more general algorithm in Section \ref{section examples}.
\end{remark}

\vspace{1em}

We now discuss a connection with a definition introduced and investigated by Gheibi, Jorgensen, and Takahashi in \cite{quasipdim}.

\begin{definition}
Let $M$ be an $R$-module. A \emph{quasi-projective resolution} of $M$ is a complex of projective $R$-modules
$$P = \xymatrix{\cdots \ar[r] & P_2 \ar[r] & P_1 \ar[r] & P_0 \ar[r] & 0}$$
such that for each $i \geqslant 0$, $H_i(P) =M^{\oplus r_i}$ for some $r_i \geqslant 0$, not all equal to zero. The module $M$ has \emph{finite quasi-projective dimension} if there exists a quasi-projective resolution of $M$ with $P_i = 0$ for $i \gg 0$.
\end{definition}

\begin{question}\label{q:GJT} (Gheibi--Jorgensen--Takahashi \cite[Question 3.12]{quasipdim}) 
If every finitely generated $R$-module has finite quasi-projective dimension, is $R$ a complete intersection?
\end{question}

In \cite[Corollary 3.8]{quasipdim}, it is shown that if $R$ is  complete intersection, then every finitely generated $R$-module has finite quasi-projective dimension. Furthermore, every module of finite quasi-projective dimension is proxy small, see \cite[Proposition 3.11]{quasipdim}; however, finite quasi-projective dimension is not equivalent to a module being proxy small, as shown in \cite[Example 4.9]{quasipdim}. Regardless, Theorem \ref{t:characterization} answers Question \ref{q:GJT} in the affirmative in the following setting.

\begin{corollary}
Whenever $R$ has large enough cohomological support, then from Theorem \ref{t:characterization}, there exists a finite length $R$-module that has infinite quasi-projective dimension. Moreover, when the residue field is infinite there exists a singular quotient of $R$ which is an artinian hypersurface of infinite quasi-projective dimension over $R$. 
\end{corollary}

\begin{corollary}
If $R$ is an equipresented local ring, then the following are equivalent:
 \begin{enumerate}
     \item $R$ is a complete intersection; 
     \item every finitely generated $R$-module has finite quasi-projective dimension;
    \item every finite length $R$-module has finite quasi-projective dimension.
    \end{enumerate}
    Moreover, when  $k$ is infinite then these are equivalent to:
\begin{enumerate}
 \setcounter{enumi}{3}
     \item each singular artinian hypersurface which is a quotient of $R$ has finite quasi-projective dimension.
 \end{enumerate}
\end{corollary}

\section{Examples}\label{section examples}

\begin{example}\label{example truncated}
    Let $Q$ be a regular local ring with maximal ideal $\m = (x_1, \ldots, x_d)$, and consider $R = Q/\m^s$ for some $s \geqslant 2$. As $R$ is an equipresented non-complete intersection, we know from Theorem \ref{t:characterization} that there exists a finite length $R$-module that is not proxy small over $R$. The point of this example is that even without the assumption that the residue field is infinite we can  explicitly construct a single artinian hypersurface quotient of $R$ that is not proxy small. 
  
    Indeed, define the $Q$-module $M$ to be
    \[
    M\coloneqq Q/(x_1^s,x_2,\ldots, x_d).
    \]
 It is evident that 
 \[
    \m^s\subseteq (x_1^s,x_2,\ldots, x_d)
    \]
     and  $M$ is an artinian singular hypersurface. 
    Therefore, by Lemma \ref{l:ben},
    $\V_R(M)$ is a  hyperplane in $\V_R$.
  However,  $\pd_{Q/f}R=\infty$ for any $f\in \m^s$  (see, for example,  \cite[10.3.8]{IFR}) and hence $\V_R(R)=\V_R$. Thus, $M$ is a singular artinian hypersurface quotient of $R$ that is not proxy small over $R$. 
\end{example}

\begin{example}\label{exampleshortgor}
Let $k$ be any field, $e \geqslant 3$, $Q = k \lb x_1,\ldots, x_e \rb$, and let $I$ be the ideal generated by 
\[
\{x_1^2-x_i^2:2 \leqslant i \leqslant e\} \cup \{x_ix_j:1 \leqslant i<j\leqslant e\}.
\]The ring $R=Q/I$ is well-known to be Gorenstein but not a complete intersection and $I$ is minimally generated by the quadratics listed above. As $R$ is an equipresented non-complete intersection, by Corollary \ref{corep} we conclude that there exists a non-proxy small module over $R$. 
In fact, more can be said in this case. Namely, we claim that \[
\V_R(R)=\V_R
\] 
is full, and so each test module constructed in using Algorithm \ref{alg} fails to be proxy small. 

Indeed,  let $A$ be the polynomial ring $k [ x_1,\ldots, x_e ]$. In \cite[Example 9.14]{DGI}, the authors show that for any element $f$ of $A$ contained in the ideal  \[
J = A \{x_1^2-x_i^2:2 \leqslant i \leqslant e\} + A \{x_ix_j:1 \leqslant i<j \leqslant e\},
\] $\pd_{A/(f)} (A/J) = \infty.$ By completing at $(x_1,\ldots, x_e)$ it follows 
\[
\pd_{Q/(f)} (R) = \infty
\] for any $f$ in $I$ that is in the image of the completion $A\to Q$ mapping $J$ into $I$. Therefore, 
 $\pd_{Q/(f)} (R) =\infty$ for any $f$ that is a nonzero $k$-linear combination of the generators for $I$ and hence,  
 \[
 \V_R(R)=\V_R
 \]
 as claimed. 

To give an explicit illustration of how Algorithm \ref{alg} works, we consider the case when $e=3$. Following Algorithm \ref{alg}, we produce  modules $M_i = R/J_i$ defined by the ideals
$$J_1 = (x^2-y^2,y-z,x), \quad J_2 = (y^2-z^2,y,x), \quad J_3 = (xy,x-y,x-z)$$
$$J_4 = (x^2-y^2 + yz - xy, y-z, x-y-z) \quad \textrm{and} \quad J_5 = (xy-xz,y,x-z).$$
A priori, just considering $R$ as an equipresented non-complete intersection, all we know is that at least one of these is not proxy small. However, as discussed above \[\V_R(R)=\V_R=k^5\] and hence, each $M_i$ is {not} proxy small over $R$.  Furthermore, in this example, we do not need the assumption that $k$ is infinite to construct the singular artinian hypersurface quotients $M_i$ over $R$ that are not proxy small over $R$. Finally, it is also worth remarking that a similar conclusion was made in \cite[9.14]{DGI}, but the use of cohomological supports gives a simpler argument that these quotients cannot be proxy small $R$-modules. 
\end{example}

Despite the fact that Theorem \ref{t:characterization} only applies to rings satisfying Condition (\ref{e:technical}), we can still apply our strategy in some cases that do not satisfy Condition (\ref{e:technical}), as the following example shows.

\begin{example}\label{example thomas}
Let $k$ be a field of characteristic not 2, $Q = k \lb x,y,z\rb$, and take 
$R=Q/I$, where $I = (x^2+y^2+z^2,xyz,x^3)$. We can make use of our techniques to find a module that is not proxy small; however, $R$ does not satisfy (\ref{e:technical}), and thus is not covered by Theorem \ref{t:characterization}. In fact, $n-c = 3-1 = 2$, and on the other hand $R$ has an embedded deformation (cf. Example \ref{examplemaximalsupport}(\ref{examplesmallcodepth})) defined by $x^2+y^2+z^2$.  Thus, by \cite[5.3.6]{Pol2}, $\V_R(R)$ is a hyperplane in $\V_R=k^3$, and so it is a $2$-dimensional subspace. 

Consider the following ideals in $Q$:
\[
J_1 = (x^2+y^2+z^2, y, x^3) \textrm{ and } J_2 = (x^2-2z, xyz, y+z).
\]
First, note that $I \supseteq J_1$ and $I \supseteq J_2$, and that both $M_1 = Q/J_1$ and $M_2 = Q/J_2$ are artinian codimension two complete intersection rings. Moreover, since both $x^2+y^2+z^2$ and $x^3$ are minimal generators of $J_1$, and $xyz$ is not, the kernel $K_1$ of the $k$-vector space map $I/\m I \longrightarrow J_1 /\m J_1$ has dimension $1$, and it is generated by the image of $xyz$ in $I/\m I$. In contrast, $xyz$ is a minimal generator of $J_2$, and thus $K_1 \cap K_2 = 0$, where $K_2$ is the kernel of the map corresponding to $J_2$. By Lemma \ref{l:ben},
\[
V_R(M_1) \cap V_R(M_2) = 0,
\]
and therefore one of these two modules is not proxy small. However, since $\V_R(M_i)$ is a $1$-dimensional subspace of $\V_R$, it follows from Fact \ref{var}(\ref{var1}) that \emph{both} $M_1$ and $M_2$ fail to be proxy small over $R$.
\end{example}

\begin{example}
Consider $Q=k[ x,y,z,w]$, $I=(x^4,xy,yz,zw,w^3)$ and $R=Q/I$. According to Macaulay2 \cite{M2} computations, $\V_R(R)$ is the union of two hyperplanes in $\V_R=k^5$:
\[
\V_R(R)=\{(a_1,\ldots,a_5)\in k^5: a_1=0\text{ or }a_5=0\}.
\]
Note that $R$ is not equipresented but $R$ has spanning support, so any minimal generator of order two gives rise to a non-proxy small module via Algorithm \ref{alg}. For example, considering the minimal generator $yz$ we obtain the non-proxy small module 
\[M=Q/(yz,x,w,y-z).
\]The algorithm does not apply to minimal generators of order 3 or 4, yet we can still produce non-proxy small modules corresponding to, for example, $x^4$ or $w^3$: namely, $M_1=Q/(x^4,y,z,w)$ and $M_2=Q/(w^3,x,y,z)$, respectively. Indeed, one can directly check that
\[
\V_R(M_1)=\{(a_1,\ldots,a_5)\in k^5: a_1=0\}\text{ and }\V_R(M_2)=\{(a_1,\ldots,a_5)\in k^5: a_5=0\},
\]  and $(x^4,y,z,w)\supseteq I$ and $(w^3,x,y,z)\supseteq I$, justifying these are not proxy small $R$-modules. 
 \end{example}

\begin{example}[Stanley-Reisner rings]\label{example monomial}
    Let $k$ be any field, $Q = k[x_1, \ldots, x_d]$, and let $I = (f_1, \ldots, f_n) \subseteq (x_1, \ldots, x_d)^2$ be a monomial ideal in $Q$, minimally generated by monomials $f_1, \ldots, f_n$. Assume that $R = Q/I$ is not a complete intersection. If $I$ is squarefree, then we can always find an Artinian quotient of $R$ that is not proxy small, independently of whether $I$ satisfies the hypothesis of Theorem \ref{t:characterization}. In fact, the process we will describe works as long as we assume that the supports\footnote{The support of a monomial $f$ is the set of variables that appear in $f$ with nonzero coefficient.} of any two of $f_1, \ldots, f_n$ are incomparable.
    
    Fix one of those $f_i$, say $f = f_1 = x_1^{a_1} \cdots x_d^{a_d}$, and assume without loss of generality that $a_1 \neq 0$. Then consider the ideal
    $$J = \left(f, x_1 - x_i, x_j \, \left|\right. \textrm{ for all } i, j \textrm{ such that } a_i \neq 0, a_j = 0 \right).$$
This ideal $J$ has some useful properties.
    \begin{enumerate}[label = \arabic*),leftmargin=1.5em]
    
        \item Our given set of generators for $J$ is a regular sequence.
        
        We gave $d$ generators, and $Q/J \cong k[x_1]/(x_1^{a_1+\cdots+a_n})$ has dimension $0$.
        
        \item $I \subseteq J$.
        
        By assumption, the support of each of the monomials $f_2, \ldots, f_n$ contains a variable $x_j$ not in the support of $f_1$, and thus $f_2, \ldots, f_n \in J$. 
        
        \item $f_2, \ldots, f_n \in (x_1, \ldots, x_d) J$.
        
        Each of these monomials is in $(x_j)$ for some $j$ with $a_j=0$, and has degree at least $2$. 
        
        \item $f \notin (x_1, \ldots, x_d) J$.
        
        It's enough to check that $f \notin (x_1 - x_i \, | \, a_i \neq 0)$, which is immediate once we set all variables to $1$.
    \end{enumerate}
    Now if we follow this recipe and construct $J_1, \ldots, J_n$ for each $f_1, \ldots, f_n$, each one of these ideals contains exactly one of the $f_i$ as a minimal generator, and thus
    \[
    \bigcap_{i = 1}^n \ker \V_{R\to Q/J_i} = 0\,.
    \]
    By Strategy \ref{strategy} and Lemma \ref{l:ben}, one of $Q/J_1, \ldots, Q/J_n$ is not a proxy small module.
\end{example}

Finally, what are examples of rings where we \emph{cannot} apply our strategy as of yet? A minimal example of a ring not satisfying (\ref{e:technical}) would be presented by an ideal $I$ with $3$ minimal generators of different $\m$-adic orders, and such that
$$\dim_k \left({\rm span} \V_R(R)\right) = 1\,,$$
since in that setting we would have
$$\dim_k \left( \dfrac{\m^{d+1}\cap I}{\m I} \right) \leqslant 1 = \dim_k \left( {\rm span} \V_R(R) \right)\,.$$
Note, however, that even given such a ring, it is not immediate that our strategy wouldn't apply --- we just have not yet proven that it does. However, we have no examples of such rings.

\begin{question}
Is there a non-complete intersection $R$ with $\dim_k\left( {\rm span} \V_R(R)\right)=1$? That is, can $\V_R(R)$ be a line?
\end{question}

By \cite[5.3.6]{Pol}, no such examples can exist when the codepth of $R$ is less than 4. Furthermore, the investigations in this article seem to suggest that the generic variety $\V_R(R)$ tends to be ``large," in the sense that it has small codimension, when $R$ is not a complete intersection.

\section*{Acknowledgements}
This work was started while the second author was on a year long visit to the University of Utah; she thanks the University of Utah Math Department for their hospitality. The authors thank Thomas Polstra for suggesting examples which inspired Example \ref{example thomas}. The authors also thank Claudia Miller and Craig Huneke for helpful conversations, and Janina Letz, Jian Liu,  Peder Thompson, and Srikanth Iyengar for their helpful comments on an earlier draft of the paper.

Macaulay2 computations \cite{M2} were crucial throughout this project. The second author was supported by NSF grant DMS-2001445. The third author was supported by NSF grant DMS-2002173 and NSF RTG grant DMS-1840190.

\bibliographystyle{plain}

\end{document}